\documentclass{article}
\usepackage{amssymb}
\usepackage{amsmath}
\usepackage{amsthm}

\newcommand{\ie}{i.e.}

\newcommand{\nsf}{n.s.f.}
\newcommand{\GNS}{G.N.S.}
\newcommand{\G}{\mathbb{G}}
\newcommand{\Hq}{\mathbb{H}}
\newcommand{\irr}{\text{Irred}(\G)}
\newcommand{\irrH}{\text{Irred}(\mathbb{H})}
\newcommand{\irrHi}{\text{Irred}(\mathbb{H}_{i})}

\newcommand{\bh}{\mathcal{B}(H)}
\newcommand{\bhx}{\mathcal{B}(H_{x})}
\newcommand{\bk}{\mathcal{B}(K)}
\newcommand{\ot}{\otimes}
\newcommand{\id}{\text{id}}
\newcommand{\Z}{\mathbb{Z}}
\newcommand{\R}{\mathbb{R}}
\newcommand{\N}{\mathbb{N}}
\newcommand{\C}{\mathbb{C}}

\newcommand{\V}{\mathbb{V}}
\newcommand{\Gh}{\widehat{\mathbb{G}}}
\newcommand{\Hh}{\widehat{\mathbb{H}}}
\newcommand{\dimq}{\text{dim}_{\text{q}}(x)}

\newcommand{\coGh}{c_{0}(\Gh)}
\newcommand{\coHh}{c_{0}(\Hh)}
\newcommand{\linfG}{l^{\infty}(\Gh)}
\newcommand{\linfH}{l^{\infty}(\Hh)}
\newcommand{\linfGmodH}{l^{\infty}(\Gh/\Hh)}
\newcommand{\ldeGmodH}{l^{2}(\Gh/\Hh)}

\theoremstyle{plain}
\newtheorem{theorem}{Theorem}
\newtheorem{lemma}{Lemma}
\newtheorem{proposition}{Proposition}
\newtheorem{corollary}{Corollary}
\theoremstyle{definition}
\newtheorem{definition}{Definition}
\newtheorem{example}{Example}
\newtheorem{notation}{Notation}
\newtheorem{remark}{Remark}

\begin{document}

\title{\normalsize{\textbf{KAZHDAN'S PROPERTY $T$ FOR DISCRETE QUANTUM GROUPS}}}
\author{\normalsize{\textsc{Pierre Fima}}\footnote{Department of Mathematics, University of Illinois at Urbana-Champaign, Urbana, Illinois 61801, United States. Email: pfima@illinois.edu}\,\,\footnote{This research was supported by the ANR (NCLp).}}
\date{}
\maketitle

\begin{abstract}
We give a simple definition of property $T$ for discrete quantum groups. We prove the basic expected properties: discrete quantum groups with property $T$ are finitely generated and unimodular. Moreover we show that, for ``I.C.C.'' discrete quantum groups, property $T$ is equivalent to Connes' property $T$ for the dual von Neumann algebra. This allows us to give the first example of a property $T$ discrete quantum group which is not a group using the twisting construction.
\end{abstract}

\section{Introduction}

In the 1980's, Woronowicz \cite{Wor1}, \cite{Wor2}, \cite{Wor3} introduced the notion of a compact quantum group and generalized the classical Peter-Weyl representation theory. Many interesting examples of compact quantum groups are available by now: Drinfel'd and Jimbo \cite{Drinfeld}, \cite{Jimbo} introduced q-deformations of compact semi-simple Lie groups, and Rosso \cite{Rosso} showed that they fit into the theory of Woronowicz. Free orthogonal and unitary quantum groups were introduced by Van Daele and Wang \cite{Wang} and studied in detail by Banica \cite{Banica1}, \cite{Banica2}.

Some discrete group-like properties and proofs have been generalized to (the dual of) compact quantum groups. See, for example, the work of Tomatsu \cite{Tom} on amenability, the work of Banica and Vergnioux \cite{BanVerg} on growth and the work of Vergnioux and Vaes \cite{VaesVerg} on boundary.

The aim of this paper is to define property $T$ for discrete quantum groups. We give a definition analogous to the group case using almost invariant vectors. We show that a discrete quantum group with property T is finitely generated, \ie{} the dual is a compact quantum group of matrices. Recall that a locally compact group with property $T$ is unimodular. We show that the same result holds for discrete quantum groups, \ie{} every discrete quantum group with property T is a Kac algebra. In \cite{CJ} Connes and Jones defined property $T$ for arbitrary von Neumann algebras and showed that an I.C.C. group has property $T$ if and only if its group von Neumann algebra (which is a $\rm{II}_{1}$ factor) has property $T$. We show that if the group von Neumann algebra of a discrete quantum group $\Gh$ is an infinite dimensional factor (\ie{} $\Gh$ is ``I.C.C.''), then $\widehat{\G}$ has property $T$ if and only if its group von Neumann algebra is a $\rm{II}_{1}$ factor with property $T$. This allows us to construct an example of a discrete quantum group with property $T$ which is not a group by twisting an I.C.C. property $T$ group. In addition we show that free quantum groups do not have property $T$.

This paper is organized as follows: in Section 2 we recall the notions of compact and discrete quantum groups and the main results of this theory. We introduce the notion of discrete quantum sub-groups and prove some basic properties of the quasi-regular representation. We also recall the definition of property $T$ for von Neumann algebras. In Section 3 we introduce property $T$ for discrete quantum groups, we give some basic properties and we show our main result.

\section{Preliminaries}

\subsection{Notations}
The scalar product of a Hilbert space $H$, which is denoted by $\langle .,.\rangle$, is supposed to be linear in the first variable. The von Neumann algebra of bounded operators on $H$ will by denoted by $\bh$ and the $C^{*}$ algebra of compact operators by $\mathcal{B}_{0}(H)$. We will use the same symbol $\ot$ to denote the tensor product of Hilbert spaces, the minimal tensor product of $C^{*}$ algebras and the spatial tensor product of von Neumann algebras. We will use freely the leg numbering notation.

\subsection{Compact quantum groups}

We briefly overview the theory of compact quantum groups developed by Worono\-wicz in \cite{Wor3}. We refer to the survey paper \cite{VD} for a smooth approach to these results.

\begin{definition}
A \textit{compact quantum group is a pair} $\G=(A,\Delta)$, where $A$ is a unital $C^{*}$ algebra; $\Delta$ is unital *-homomorphism from $A$ to $A\ot A$ satisfying $(\Delta\ot\id)\Delta=(\id\ot\Delta)\Delta$ and $\Delta(A)(A\ot 1)$ and $\Delta(A)(1\ot A)$ are dense in $A\ot A$.
\end{definition}
\begin{notation}
We denote by $C(\G)$ the $C^{*}$ algebra $A$.
\end{notation}

The major results in the general theory of compact quantum groups are the existence and uniqueness of the Haar state and the Peter-Weyl representation theory.
\begin{theorem}
Let $\G$ be a compact quantum group. There exists a unique state $\varphi$ on $C(\G)$ such that $(\id\ot\varphi)\Delta(a)=\varphi(a)1=(\varphi\ot\id)\Delta(a)$ for all $a\in C(\G)$. The state $\varphi$ is called the Haar state of $\G$.
\end{theorem}
\begin{notation}
The Haar state need not be faithful. We denote by $\G_{\text{red}}$ the \textit{reduced quantum group} obtained by taking $C(\G_{\text{red}})=C(\G) / I$ where $I=\{x\in A\,|\,\varphi(x^{*}x)=0\}$. The Haar measure is faithful on $\G_{\text{red}}$. We denote by $L^{\infty}(\G)$ the von Neumann algebra generated by the G.N.S. representation of the Haar state of $\G$. Note that $L^{\infty}(\G_{\text{red}})=L^{\infty}(\G)$.
\end{notation}

\begin{definition}
A \textit{unitary representation} $u$ of a compact quantum group $\G$ on a Hilbert space $H$ is a unitary element $u\in$M$(\mathcal{B}_{0}(H)\ot C(\G))$ satisfying
$$(\id\ot\Delta)(u)=u_{12}u_{13}.$$
Let $u^{1}$ and $u^{2}$ be two unitary representations of $\G$ on  the respective Hilbert spaces $H_{1}$ and $H_{2}$. We define the set of \textit{intertwiners}
$$\text{Mor}(u^{1},u^{2})=\{T\in\mathcal{B}(H_{1},H_{2})\,|\, (T\ot 1)u^{1}=u^{2}(T\ot 1)\}.$$
A unitary representation $u$ is said to be \textit{irreducible} if $\text{Mor}(u,u)=\C 1$. Two unitary representations $u^{1}$ and $u^{2}$ are said to be \textit{unitarily equivalent} if there is a unitary element in $\text{Mor}(u^{1}, u^{2})$.
\end{definition}

\begin{theorem}
Every irreducible representation is finite-dimensional. Every unitary representation is unitarily equivalent to a direct sum of irreducibles.
\end{theorem}

\begin{definition}
Let $u^{1}$ and $u^{2}$ be unitary representations of $\G$ on the respective Hilbert spaces $H_{1}$ and $H_{2}$. We define the tensor product
$$u^{1}\ot u^{2}=u^{1}_{13}u^{2}_{23}\in\text{M}(\mathcal{B}_{0}(H_{1}\ot H_{2})\ot C(\G)).$$
\end{definition}
\begin{notation}
We denote by $\irr$ the set of (equivalence classes) of irreducible unitary representations of a compact quantum group $\G$. For every $x\in\irr$ we choose representatives $u^{x}$ on the Hilbert space $H_{x}$. Whenever $x,y\in\irr$, we use $x\ot y$ to denote the (class of the) unitary representation $u^{x}\ot u^{y}$. The class of the trivial representation is denoted by $1$.
\end{notation}

The set $\irr$ is equipped with a natural involution $x\mapsto\bar{x}$ such that $u^{\bar{x}}$ is the unique (up to unitary equivalence) irreducible representation such that
$$\text{Mor}(1,x\ot\bar{x})\neq 0\neq \text{Mor}(1,\bar{x}\ot x).$$
This means that $x\ot\bar{x}$ and $\bar{x}\ot x$ contain a non-zero invariant vector. Let $E_{x}\in H_{x}\ot H_{\bar{x}}$ be a non-zero invariant vector and $J_{x}$ the invertible antilinear map from $H_{x}$ to $H_{\bar{x}}$ defined by
$$\langle J_{x}\xi,\eta\rangle=\langle E_{x}, \xi\ot \eta\rangle,\quad\text{for all}\,\,\,\xi\in H_{x},\,\,\eta\in H_{\bar{x}}.$$
Let $Q_{x}=J_{x}^{*} J_{x}$. We will always choose $E_{x}$ and $E_{\bar{x}}$ normalized such that $||E_{x}||=||E_{\bar{x}}||$ and $J_{\bar{x}}=J_{x}^{-1}$. Then $Q_{x}$ is uniquely determined, $\text{Tr}(Q_{x})=||E_{x}||^{2}=\text{Tr}(Q_{x}^{-1})$ and $Q_{\bar{x}}=(J_{x}J_{x}^{*})^{-1}$. $\text{Tr}(Q_{x})$ is called the \textit{quantum dimension} of $x$ and is denoted by $\dimq$. The unitary representation $u^{\bar{x}}$ is called the \textit{contragredient} of $u^{x}$.

The G.N.S. representation of the Haar state is given by $(L^{2}(\G),\Omega)$ where $L^{2}(\G)=\bigoplus_{x\in\irr}H_{x}\ot H_{\bar{x}}$, $\Omega\in H_{1}\ot H_{\bar{1}}$ is the unique norm one vector, and
$$(\omega_{\xi,\eta}\ot\id)(u^{x})\Omega=\frac{1}{||E_{x}||}\xi\ot J_{x}(\eta),\quad\text{for all}\,\,\,\xi,\eta\in H_{x}.$$
It is easy to see that $\varphi$ is a trace if and only if $Q_{x}=\text{id}$ for all $x\in\irr$. In this case $||E_{x}||=\sqrt{n_{x}}$ where $n_{x}$ is the dimension of $H_{x}$ and $J_{x}$ is an anti-unitary operator.

\begin{notation}
Let $C(\G)_{s}$ be the vector space spanned by the coefficients of all irreducible representations of $\G$. Then $C(\G)_{s}$ is a dense unital *-subalgebra of $C(\G)$. Let $C(\G_{\text{max}})$ be the maximal $C^{*}$ completion of the unital *-algebra $C(\G)_{s}$. $C(\G_{\text{max}})$ has a canonical structure of a compact quantum group. This quantum group is denoted by $\G_{\text{max}}$ and it is called the \textit{maximal quantum group}.
\end{notation}

A \textit{morphism of compact quantum groups} $\pi\,:\,\G\rightarrow\Hq$ is a unital *-homomor\-phism from $C(\G_{\text{max}})$  to $C(\Hq_{\text{max}})$ such that $\Delta_{\Hq}\circ\pi=(\pi\otimes\pi)\circ\Delta_{\G}$, where $\Delta_{\G}$ and $\Delta_{\Hq}$ denote the comultiplications for $\G_{\text{max}}$ and $\Hq_{\text{max}}$ respectively. We will need the following easy Lemma.

\begin{lemma}\label{LemSurj}
Let $\pi$ be a surjective morphism of compact quantum group from $\G$ to $\Hq$ and $\tilde{\pi}$ be the surjective *-homomorphism from $C(\G_{\text{max}})$  to $C(\Hq)$ obtained by composition of $\pi$ with the canonical surjection $C(\Hq_{\text{max}})\rightarrow C(\Hq)$. Then for every irreducible unitary representation $v$ of $\Hq$ there exists an irreducible unitary representation $u$ of $\G$ such that $v$ is contained in the unitary representation $(\text{id}\ot\tilde{\pi})(u)$.
\end{lemma}

\begin{proof}
Let $\varphi$ be the Haar state of $\Hq$ and $v$ be an irreducible unitary representation of $\Hq$ on the Hilbert space $H_{v}$. Because $v$ is irreducible it is sufficient to show that there exists a unitary irreducible representation $u$ of $\G$ such that Mor$(w,v)\neq\{0\}$, where $w=(\text{id}\ot\tilde{\pi})(u)$. Suppose that the statement is false. Then for all irreducible unitary representations $u$ of $\G$ on $H_{u}$, we have Mor$(w,v)=\{0\}$. By \cite{VD}, Lemma 6.3, for every operator $a\,:\, H_{v}\rightarrow H_{u}$ the operator $(\text{id}\ot\varphi)(v^{*}(a\ot 1)w)$ is in Mor$(w,v)$. It follows that for every irreducible unitary representation $u$ of $\G$ and every operator $a\,:\, H_{v}\rightarrow H_{u}$ we have $(\text{id}\ot\varphi)(v^{*}(a\ot 1)w)=0$. Using the same techniques as in \cite{VD}, Theorem 6.7, (because, by the surjectivity of $\pi$, $\tilde{\pi}(C(\G)_{s})$ is dense in $C(\Hq)$) we find $(\text{id}\ot\varphi)(v^{*}v)=0$. But this is a contradiction as $v^{*}v=1$.
\end{proof}

The collection of all finite-dimensional unitary representations (given with the concrete Hilbert spaces) of a compact quantum group $\G$ is a \textit{complete concrete monoidal} $W^{*}$-\textit{category}. We denote this category by $\mathcal{R}(\G)$. We say that $\mathcal{R}(\G)$ is finitely generated if there exists a finite subset $E\subset\irr$ such that for all finite-dimensional unitary representations $r$ there exists a finite family of morphisms $b_{k}\in\text{Mor}(r_{k},r)$, where $r_{k}$ is a product of elements of $E$, and $\sum_{k} b_{k}b_{k}^{*}=I_{r}$. It is not difficult to show that $\mathcal{R}(\G)$ is finitely generated if and only if $\G$ is a compact quantum group of matrices (see \cite{Wor2}).

\subsection{Discrete quantum groups}

A discrete quantum group is defined as the dual of a compact quantum group.

\begin{definition}
Let $\G$ be a compact quantum group. We define the dual \textit{discrete quantum group} $\widehat{\G}$ as follows:
$$c_{0}(\widehat{\G})=\bigoplus_{x\in\irr}^{c_{0}}\mathcal{B}(H_{x}),\quad l^{\infty}(\widehat{\G})=\bigoplus_{x\in\irr}^{\infty}\mathcal{B}(H_{x}).$$
We denote the minimal central projection of $l^{\infty}(\widehat{\G})$ by $p_{x}$, $x\in\irr$. We have a natural unitary $\mathbb{V}\in$M$(c_{o}(\widehat{\G})\ot C(\G))$ given by
$$\mathbb{V}=\bigoplus_{x\in\irr} u^{x}.$$
We have a natural comultiplication
$$\hat{\Delta}\,:\,l^{\infty}(\widehat{\G})\rightarrow l^{\infty}(\widehat{\G})\ot l^{\infty}(\widehat{\G})\,:\,(\hat{\Delta}\ot id)(\mathbb{V})=\mathbb{V}_{13}\mathbb{V}_{23}.$$
The comultiplication is given by the following formula
$$\hat{\Delta}(a)S=Sa,\quad\text{for all}\,\,a\in\mathcal{B}(H_{x}),\,S\in\text{Mor}(x,yz),\,x,y,z\in\irr.$$
\end{definition}
\begin{remark}
The maximal and reduced versions of a compact quantum group are different versions of the same underlying compact quantum group. This different versions give the same dual discrete quantum group, \ie{} $\Gh=\widehat{\G_{\text{red}}}=\widehat{\G_{\text{max}}}$. This means that $\Gh$, $\widehat{\G_{\text{red}}}$ and $\widehat{\G_{\text{max}}}$ have the same $C^{*}$ algebra, the same von Neumann algebra and the same comultiplication.
\end{remark}

A \textit{morphism of discrete quantum groups} $\hat{\pi}\,:\,\Gh\rightarrow\Hh$ is a non-degenerate *-homomorphism from $\coGh$  to M($\coHh$) such that $\hat{\Delta}_{\Hq}\circ\pi=(\pi\otimes\pi)\circ\hat{\Delta}_{\G}$, where $\hat{\Delta}_{\G}$ and $\hat{\Delta}_{\Hq}$ denote the comultiplication for $\Gh$ and $\Hh$ respectively. Every morphism of compact quantum groups $\pi\, :\, \G\rightarrow \Hq$ admits a canonical dual morphism of discrete quantum groups $\hat{\pi}\, :\, \Gh\rightarrow \Hh$. Conversely, every morphism of discrete quantum groups $\hat{\pi}\, :\, \Gh\rightarrow \Hh$ admits a canonical dual morphism of compact quantum groups $\pi\, :\, \G\rightarrow \Hq$. Moreover, $\pi$ is surjective (resp. injective) if and only if $\hat{\pi}$ is injective (resp. surjective).

We say that a discrete quantum group $\widehat{\G}$ is \textit{finitely generated} if the category $\mathcal{R}(\G)$ is finitely generated.

We will work with representations in the von Neumann algebra setting.

\begin{definition}
Let $\widehat{\G}$ be a discrete quantum group. A \textit{unitary representation} $U$ of $\widehat{\G}$ on a Hilbert space $H$ is a unitary $U\in l^{\infty}(\widehat{\G})\ot\bh$ such that :
$$(\hat{\Delta}\ot\id)(U)=U_{13}U_{23}.$$
\end{definition}

Consider the following maximal version of the unitary $\mathbb{V}$:
$$\mathcal{V}=\bigoplus_{x\in\irr} u^{x}\in\text{M}(c_{o}(\widehat{\G})\ot C(\G_{\text{max}})).$$

For every unitary representation $U$ of $\Gh$ on a Hilbert space $H$ there exists a unique *-homomorphism $\rho\, :\, C(\G_{\text{max}})\rightarrow\bh$ such that $(\text{id}\otimes\rho)(\mathcal{V})=U$.
\newline
\begin{notation}
Whenever $U$ is a unitary representation of $\widehat{\G}$ on a Hilbert space $H$ we write $U=\sum_{x\in\irr} U^{x}$ where $U^{x}=Up_{x}$ is a unitary in $\mathcal{B}(H_{x})\ot\bh$.
\end{notation}

The discrete quantum group $\linfG$ comes equipped with a natural modular structure. Let us define the following canonical states on $\bhx$:
$$\varphi_{x}(A)=\frac{\text{Tr}(Q_{x}A)}{\text{Tr}(Q_{x})},\quad\text{and}\quad\psi_{x}(A)=\frac{\text{Tr}(Q_{x}^{-1}A)}{\text{Tr}(Q_{x}^{-1})},\quad\text{for all}\,\,A\in\bhx.$$

The states $\varphi_{x}$ and $\psi_{x}$ provide a formula for the invariant normal semi-finite faithful (\nsf) weights on $\linfG$.

\begin{proposition}
The left invariant weight $\hat{\varphi}$ and the right invariant weight $\hat{\psi}$ on $\Gh$ are given by
$$\hat{\varphi}(a)=\sum_{x\in\irr}\dimq^{2}\varphi_{x}(ap_{x})\quad \text{and}\quad\hat{\psi}(a)=\sum_{x\in\irr}\dimq^{2}\psi_{x}(ap_{x}),$$
for all $a\in\linfG$ whenever this formula makes sense.
\end{proposition}

A discrete quantum is unimodular (\ie{} the left and right invariant weights are equal) if and only if the Haar state $\varphi$ on the dual is a trace. In general, a discrete quantum group is not unimodular, and it is easy to check that the Radon-Nikodym derivative is given by
$$[D\hat{\psi}\,:\,D\hat{\varphi}]_{t}=\hat{\delta}^{it}\quad\text{where}\quad\hat{\delta}=\sum_{x\in\irr} Q_{x}^{-2}p_{x}.$$
The positive self-adjoint operator $\hat{\delta}$ is called the \textit{modular element}: it is affiliated with $\coGh$ and satisfies $\hat{\Delta}(\hat{\delta})=\hat{\delta}\otimes\hat{\delta}$.

The following Proposition is very easy to prove.

\begin{proposition}\label{Propeigenvalues}
Let $\Gamma$ be the subset of $\R^{*}_{+}$ consisting of all the eigenvalues of the operators $Q_{x}^{-2}$ for $x\in\irr$. Then $\Gamma$ is a subgroup of $\R^{*}_{+}$ and Sp$(\widehat{\delta})=\Gamma\cup\{0\}$.
\end{proposition}

\begin{proof}
Note that, because $J_{\bar{x}}=J_{x}^{-1}$, the eigenvalues of $Q_{\bar{x}}$ are the inverse of the eigenvalues of $Q_{x}$. Using the formula $SQ_{z}=Q_{x}\ot Q_{y} S$, when $z\subset x\ot y$ and $S\in\text{Mor}(z,x\ot y)$ is an isometry, the Proposition follows immediately.
\end{proof}

\subsection{Discrete quantum subgroups}

Let $\G$ be a compact quantum group with representation category $\mathcal{C}$. Let $\mathcal{D}$ be a full subcategory such that $1_{\mathcal{C}}\in\mathcal{D}$, $\mathcal{D}\otimes\mathcal{D}\subset\mathcal{D}$ and $\overline{\mathcal{D}}=\mathcal{D}$. By the Tannaka-Krein Reconstruction Theorem of Woronowicz \cite{Wor2} we know that there exists a compact quantum group $\mathbb{H}$ such that the representation category of $\mathbb{H}$ is $\mathcal{D}$. We say that $\Hh$ is a \textit{discrete quantum subgroup} of $\Gh$. We have $\irrH\subset\irr$. We collect some easy observations in the next proposition. We denote by a subscript $\Hq$ the objects associated to $\Hq$.
\newpage

\begin{proposition}
Let $p=\sum_{x\in\irrH} p_{x}$. We have:
\begin{enumerate}
\item $\hat{\Delta}(p)(p\ot 1)=p\ot p$;
\item $l^{\infty}(\Hh)=p(l^{\infty}(\Gh))$;
\item $\hat{\Delta}_{\Hq}(a)=\hat{\Delta}(a)(p\ot p)$ for all $a\in\linfH$;
\item $\hat{\varphi}(p.)=\hat{\varphi}_{\Hq}$ and  $\hat{\delta}_{\Hq}=p\hat{\delta}$.
\end{enumerate}
\end{proposition}

\begin{proof}
For $x,y,z\in\irr$ such that $y\subset z\otimes x$, we denote by $p_{y}^{z\otimes x}\in\text{End}(x\otimes y)$ the projection on the sum of all sub-representations equivalent to $y$. Note that
\begin{eqnarray}\label{Delta}
\hat{\Delta}(p_{y})(p_{z}\otimes p_{x})=
\left\{\
\begin{array}{ll}
p_{y}^{z\ot x} &  \text{if}\,\,y\subset z\ot x,\\
0 & \text{otherwise}.
\end{array}\right.
\end{eqnarray}
Thus:
$$\hat{\Delta}(p)(p_{z}\ot p_{x})=\sum_{y\in \irrH,\,y\subset z\ot x} p^{z\ot x}_{y}.$$
Note that if $y\subset z\ot x$ and $y,\,z\in\irrH$ then $x\in\irrH$. It follows that:
$$\hat{\Delta}(p)(p\otimes p_{x})=
\left\{\
\begin{array}{ll}
p\ot p_{x} &  \text{if}\,\,x\in\irrH,\\
0 & \text{otherwise}.
\end{array}\right.$$
Thus, $\hat{\Delta}(p)(p\ot 1)=p\ot p$. The other assertions are obvious.
\end{proof}

We introduce the following equivalence relation on $\irr$ (see \cite{Verg}): if $x,y\in\irr$ then $x\sim y$ if and only if there exists $t\in\irrH$ such that $x\subset y\otimes t$. We define the right action of $\Hh$ on $\linfG$ by translation:
$$\alpha\,:\,\linfG\rightarrow \linfG\otimes\linfH,\quad\alpha(a)=\hat{\Delta}(a)(1\otimes p).$$
Using $\hat{\Delta}(p)(p\ot 1)=p\ot p$ and $\hat{\Delta}_{\Hq}=\hat{\Delta}(.)(p\ot p)$ it is easy to see that $\alpha$ satisfies the following equations:
$$(\alpha\otimes\id)\alpha=(\id\otimes\hat{\Delta}_{\Hq})\alpha\quad\text{and}\quad
(\hat{\Delta}\otimes\id)\alpha=(\id\otimes\alpha)\hat{\Delta}.$$
The first equality means that $\alpha$ is a right action of $\Hh$ on $\linfG$. Let $\linfGmodH$ be the set of fixed points of the action $\alpha$:
$$\linfGmodH:=\{a\in\linfG,|\,\alpha(a)=a\otimes 1\}.$$
Using the second equality for $\alpha$ it is easy to see that:
$$\hat{\Delta}(\linfGmodH)\subset \linfG\otimes\linfGmodH.$$
Thus the restriction of $\hat{\Delta}$ to $\linfGmodH$ gives an action of $\Gh$ on $\linfGmodH$. We denote this action by $\beta$.

\begin{proposition}
Let $T_{\alpha}=(\id\otimes\hat{\varphi}_{\Hq})\alpha$ be the normal faithful operator valued weight from $\linfG$ to $\linfGmodH$ associated to $\alpha$. $T_{\alpha}$ is semi-finite and there exists a unique \nsf{} weight $\theta$ on $\linfGmodH$ such that $\hat{\varphi}=\theta\circ T_{\alpha}$.
\end{proposition}

\begin{proof}
It follows from Eq. $(\ref{Delta})$ that $T_{\alpha}(p_{y})p_{z}=0$ if $z\nsim y$. Take $z\sim y$, we have:
\begin{eqnarray*}
T_{\alpha}(p_{y})p_{z}
&=&\sum_{x\in\irrH}\text{dim}_{q}(x)^{2}(\id\otimes\varphi_{x})(p_{y}^{z\otimes x})\\
&\leq &\sum_{x\in\irr}\text{dim}_{q}(x)^{2}(\id\otimes\varphi_{x})(p_{y}^{z\otimes x})\\
&=&(\id\otimes\hat{\varphi})(\hat{\Delta}(p_{y}))p_{z}=\hat{\varphi}(p_{y})p_{z}\\
&=&\text{dim}_{q}(y)^{2}p_{z}.
\end{eqnarray*}
It follows that $T_{\alpha}(p_{y})<\infty$ for all $y$. This implies that $T_{\alpha}$ is semi-finite. Note that $\alpha(\delta^{-it})=\delta^{-it}\otimes\delta_{\Hq}^{-it}$. It follows from \cite{Kus}, Proposition 8.7, that there exists a unique \nsf{} weight $\theta$ on $\linfGmodH$ such that $\hat{\varphi}=\theta\circ T_{\alpha}$.
\end{proof}

Denote by $\ldeGmodH$ the \GNS{} space of $\theta$ and suppose that $\linfGmodH\subset\mathcal{B}(\ldeGmodH)$. Let $U^{*}\in\linfG\otimes\mathcal{B}(\ldeGmodH)$ be the unitary implementation of $\beta$ associated to $\theta$ in the sense of \cite{Vaes}. Then $U$ is a unitary representation of $\Gh$ on $\ldeGmodH$ and $\beta(x)=U^{*}(1\otimes x)U$. We call $U$ the \textit{quasi-regular} representation of $\Gh$ modulo $\Hh$.

\begin{lemma}\label{Hinvariant}
We have $p\in\linfGmodH\cap\mathcal{N}_{\theta}$. Put $\xi=\Lambda_{\theta}(p)$. If $\Gh$ is unimodular then $U^{x}\eta\ot\xi=\eta\ot\xi$ for all $x\in\irrH$ and all $\eta\in H_{x}$.
\end{lemma}

\begin{proof}
Using $\hat{\Delta}(p_{1})(1\ot p_{x})=p_{1}^{\bar{x}\ot x}$ it is easy to see that $T_{\alpha}(p_{1})=p$. It follows that $p\in\linfGmodH$ and $\theta(p)=\hat{\varphi}(p_{1})=1$. Thus $p\in\mathcal{N}_{\theta}$. Let $x\in M^{+}$ such that $T_{\alpha}(x)<\infty$, $\omega\in\linfG_{*}^{+}$ and $\mu$ a \nsf{} weight on $\linfGmodH$. Using $(\hat{\Delta}\otimes\id)\alpha=(\id\otimes\alpha)\hat{\Delta}$ we find:
\begin{eqnarray}\notag
(\omega\otimes\mu)\beta(T_{\alpha}(x))
&=&(\omega\otimes\mu)\hat{\Delta}(T_{\alpha}(x))
=(\omega\otimes\mu)\hat{\Delta}((\id\otimes\hat{\varphi}_{\Hq})(\alpha(x)))\\\notag
&=&(\omega\otimes\mu\ot\hat{\varphi}_{\Hq})((\hat{\Delta}\otimes\id)\alpha(x))\\\notag
&=&(\omega\otimes\mu\ot\hat{\varphi}_{\Hq})((\id\otimes\alpha)\hat{\Delta}(x))\\\label{Eq0}
&=&(\omega\otimes\mu\circ T_{\alpha})\hat{\Delta}(x).
\end{eqnarray}
It follows that, for all $\omega\in\linfG_{*}^{+}$ and all $y\in\linfG^{+}$ such that $T_{\alpha}(y)<\infty$, we have:
$$(\omega\ot\theta)\beta(T_{\alpha}(y))=(\omega\ot\hat{\varphi})(\hat{\Delta}(y))=\hat{\varphi}(y)\omega(1)=\theta(T_{\alpha}(y))\omega(1).$$
Let $x\in\linfGmodH^{+}$. Because $T_{\alpha}$ is a faithful and semi-finite, there exists an increasing net of positive elements $y_{i}$ in $\linfG^{+}$ such that $T_{\alpha}(y_{i})<\infty$ for all $i$ and $\text{Sup}_{i}(T_{\alpha}(y_{i}))=x$. It follows that:
$$(\omega\ot\theta)\beta(x)=\text{Sup}((\omega\ot\theta)\beta(T_{\alpha}(y_{i})))=\text{Sup}(\theta(T_{\alpha}(y_{i})))\omega(1)=\theta(x)\omega(1),$$
for all $\omega\in\linfG_{*}^{+}$. This means that $\theta$ is $\beta$-invariant. Using this invariance we define the following isometry:
$$V^{*}(\hat{\Lambda}(x)\ot\Lambda_{\theta}(y))=(\hat{\Lambda}\ot\Lambda_{\theta})(\beta(y)(x\ot 1)).$$
Because $\Gh$ is unimodular we know from \cite{Vaes}, Proposition 4.3, that $V^{*}$ is the unitary implementation of $\beta$ associated to $\theta$ \ie{} $V=U$. Using $\hat{\Delta}(p)(p\ot 1)=p\ot p$, it follows that, for all $x\in\mathcal{N}_{\hat{\varphi}}$, we have:
$$U^{*}(p\hat{\Lambda}(x)\ot\Lambda_{\theta}(p))=(\hat{\Lambda}\ot\Lambda_{\theta})(\hat{\Delta}(p)(px\ot 1))
=p\hat{\Lambda}(x)\ot\Lambda_{\theta}(p).$$
This concludes the proof.
\end{proof}

\begin{remark}
For general discrete quantum groups it can be proved, as in \cite{Enock}, Th\'{e}or\`{e}me 2.9, that $V^{*}$ is a unitary implementing the action $\beta$ and, as in \cite{Vaes}, Proposition 4.3, that $V^{*}$ is the unitary implementation of $\beta$ associated to $\theta$. Thus the previous lemma is also true for general discrete quantum groups.
\end{remark}

\begin{lemma}\label{Invariant}
Suppose that $U$ has a non-zero invariant vector $\xi\in\ldeGmodH$. Then $\irr/\irrH$ is a finite set.
\end{lemma}

\begin{proof}
Let $\xi\in\ldeGmodH$ be a normalized $U$-invariant vector. Using $\beta(x)=U^{*}(1\otimes x)U$ it is easy to see that $\omega_{\xi}$ is a $\beta$-invariant normal state on $\linfGmodH$, \ie{} $(\id\otimes\omega_{\xi})\beta(x)=\omega_{\xi}(x)1$ for all $x\in\linfGmodH$. Let $s$ be the support of $\omega_{\xi}$ and $e=1-s$. Let $\omega$ be a faithful normal state on $\linfG$. Because the support of $\omega\otimes\omega_{\xi}$ is $1\otimes s$ and $(\omega\otimes\omega_{\xi})\beta(e)=\omega_{\xi}(e)=0$ we find $\hat{\Delta}(e)=\beta(e)\leq1\otimes e$. It follows from \cite{VaesPhD}, Lemma 6.4, that $e=0$ or $e=1$. Because $\xi$ is a non-zero vector we have $e=0$. Thus $\omega_{\xi}$ is faithful. Let $x\in M^{+}$ such that $T_{\alpha}(x)<\infty$. By Eq. $(\ref{Eq0})$ we have:
$$
(\omega\ot\omega_{\xi}\circ T_{\alpha})(\hat{\Delta}(x))
=(\omega\ot\omega_{\xi})\beta(T_{\alpha}(x))
=\omega_{\xi}(T_{\alpha}(x))\omega(1),
$$
for all $\omega\in\linfG_{*}^{+}$. Because $T_{\alpha}$ is \nsf{}, it follows easily that $\omega_{\xi}\circ T_{\alpha}$ is a left invariant \nsf{} weight on $\Gh$. Thus, up to a positive constant, we have $\omega_{\xi}\circ T_{\alpha}=\hat{\varphi}$.

Suppose that $\irr/\irrH$ is infinite, and let $x_{i}\in\irr$, $i\in\N$ be a complete set of representatives of $\irr/\irrH$. Let $a$ be the positive element of $\linfG$ defined by $a=\sum_{i\geq 0}\frac{1}{\text{dim}_{q}(x_{i})^{2}}p_{x_{i}}$. Then we have $\hat{\varphi}(a)=+\infty$ and $T_{\alpha}(a)=\sum_{i}\sum_{x\simeq x_{i}}p_{x}=1<\infty$, which is a contradiction.
\end{proof}

\subsection{Property $T$ for von Neumann algebras}

Here we recall several facts from \cite{CJ}. If $M$ and $N$ are von Neumann algebras then a correspondence from $M$ to $N$ is a Hilbert space $H$ which is both a left $M$-module and a right $N$-module, with commuting normal actions $\pi_{l}$ and $\pi_{r}$ respectively. The triple $(H,\pi_{l},\pi_{r})$ is simply denoted by $H$ and we shall write $a\xi b$ instead of $\pi_{l}(a)\pi_{r}(b)\xi$ for $a\in M$, $b\in N$ and $\xi\in H$. We shall denote by $\mathcal{C}(M)$ the set of unitary equivalence classes of correspondences from $M$ to $M$. The standard representation of $M$ defines an element $L^{2}(M)$ of $\mathcal{C}(M)$, called the identity correspondence.

Given $H\in\mathcal{C}(M)$, $\epsilon>0$, $\xi_{1},\ldots,\xi_{n}\in H$, $a_{1},\ldots,a_{p}\in M$, let $\mathcal{V}_{H}(\epsilon,\xi_{i},a_{i})$ be the set of $K\in\mathcal{C}(M)$ for which there exist $\eta_{1},\ldots,\eta_{n}\in K$ with
$$|\langle a_{j}\eta_{i}a_{k},\eta_{i^{'}}\rangle -\langle a_{j}\xi_{i}a_{k},\xi_{i^{'}}\rangle|<\epsilon,\quad\text{for all}\,\,i,i^{'},j,k.$$
Such sets form a basis of a topology on $\mathcal{C}(M)$ and, following \cite{CJ}, $M$ is said to have property $T$ if there is a neighbourhood of the identity correspondence, each member of which contains $L^{2}(M)$ as a direct summand.

When $M$ is a $\rm{II}_{1}$ factor the property $T$ is easier to understand. A $\rm{II}_{1}$ factor $M$ has property $T$ if we can find $\epsilon>0$ and $a_{1},\ldots,a_{p}\in M$ satisfying the following condition: every $H\in\mathcal{C}(M)$ such that there exists $\xi\in H$, $||\xi||=1$, with $||a_{i}\xi-\xi a_{i}||<\epsilon$ for all $i$, contains a non-zero central vector $\eta$ (\ie{} $a\eta=\eta a$ for all $a\in M$). We recall the following Proposition from \cite{CJ}.

\begin{proposition}\label{propCJ}
If $M$ is a $\rm{II}_{1}$ factor with property $T$ then there exist $\epsilon>0$, $b_{1},\ldots,b_{m}\in M$ and $C>0$ with the following property: for any $\delta\leq\epsilon$, if $H\in\mathcal{C}(M)$ and $\xi\in H$ is a unit vector satisfying $||b_{i}\xi-\xi b_{i}||<\delta$ for all $1\leq i\leq m$, then there exists a unit central vector $\eta\in H$ such that $||\xi-\eta||<C\delta$.
\end{proposition}

It is proved in \cite{CJ} that a discrete I.C.C. group has property $T$ if and only if the group von Neumann algebra $\mathcal{L}(G)$ has property $T$.

\section{Property $T$ for Discrete Quantum Groups}
\begin{definition}
Let $\widehat{\G}$ be a discrete quantum group.
\begin{itemize}
\item Let $E\subset\irr$ be a finite subset, $\epsilon >0$ and $U$ a unitary representation of $\widehat{\G}$ on a Hilbert space $K$. We say that $U$ has an $(E,\epsilon)$\textit{-invariant vector} if there exists a unit vector $\xi\in K$ such that for all $x\in E$ and $\eta\in H_{x}$ we have:
$$||U^{x}\eta\ot\xi - \eta\ot\xi ||<\epsilon||\eta ||.$$
\item We say that $U$ has \textit{almost invariant vectors} if, for all finite subsets $E\subset\irr$ and all $\epsilon >0$, $U$ has an $(E,\epsilon)$-invariant vector.
\item We say that $\widehat{\G}$ has \textit{property} $T$ if every unitary representation of $\hat{\G}$ having almost invariant vectors has a non-zero invariant vector.
\end{itemize}
\end{definition}
\begin{remark}
Let $\G=(C^{*}(\Gamma),\Delta)$, where $\Gamma$ is a discrete group and $\Delta(g)=g\ot g$ for $g\in\Gamma$. It follows from the definition that $\Gh$ has property $T$ if and only if $\Gamma$ has property $T$.
\end{remark}

The next proposition will be useful to show that the dual of a free quantum group does not have property $T$.

\begin{proposition}\label{Morph}
Let $\G$ and $\mathbb{H}$ be compact quantum groups. Suppose that there is a surjective morphism of compact quantum groups from $\G$ to $\mathbb{H}$ (or an injective morphism of discrete quantum groups from $\Hh$ to $\Gh$). If $\widehat{\G}$ has property $T$ then $\widehat{\mathbb{H}}$ has property $T$.
\end{proposition}

\begin{proof}
We can suppose that $\G=\G_{\text{max}}$ and $\Hq=\Hq_{\text{max}}$. We will denote by a subscript $\G$ (resp. $\Hq$) the object associated to $\G$ (resp. $\Hq$). Let $\pi$ be the surjective morphism  from $C(\G)$ to $C(\Hq)$ which intertwines the comultiplications. Let $U$ be a unitary representation of $\widehat{\mathbb{H}}$ on a Hilbert space $K$ and suppose that $U$ has almost invariant vectors. Let $\rho$ be the unique morphism from $C(\Hq)$ to $\bk$ such that $(\id\otimes\rho)(\mathcal{V}_{\mathbb{H}})=U$. Consider the following unitary representation of $\widehat{\G}$ on $K$:  $V=(\id\ot(\rho\circ\pi))(\mathcal{V}_{\G})$. We will show that $V$ has almost invariant vectors. Let $E\subset \irr$ be a finite subset and $\epsilon>0$. For $x\in\irr$ and $y\in\text{Irred}(\mathbb{H})$ denote by $u^{x}\in\mathcal{B}(H_{x})\ot C(\G)$ and $v^{y}\in\mathcal{B}(H_{y})\ot C(\Hq)$ a representative of $x$ and $y$ respectively. Note that $w^{x}=(\id\ot\pi)(u^{x})$ is a finite dimensional unitary representation of $\mathbb{H}$, thus we can suppose that $w^{x}=\oplus\, n_{x,y}v^{y}$. Let $L=\{y\in\text{Irred}(\mathbb{H})\,|\,\exists x\in E,\,n_{x,y}\neq 0\}$. Because $U$ has almost invariant vectors, there exists a norm one vector $\xi\in K$ such that $||U^{y}\eta\ot\xi-\eta\ot\xi||<\epsilon||\eta||$ for all $y\in L$ and all $\eta\in H_{y}$. Using the isomorphism
$$H_{x}=\bigoplus_{y\in\irrH ,\,\, n_{x,y}\neq 0}\underbrace{(H_{y}\oplus\ldots\oplus H_{y})}_{n_{x,y}},$$
we can identify $V^{x}$ with $\oplus\, n_{x,y} U^{y}$ in $\bigoplus_{y}\mathcal{B}(H_{y})\oplus\mathcal{B}(H_{y})\oplus\ldots\oplus \mathcal{B}(H_{y})\ot\mathcal{B}(K)$. With this identification it is easy to see that, for all $x\in E$ and all $\eta$ in $H_{x}$, we have $||V^{x}\eta\ot\xi-\eta\ot\xi||<\epsilon||\eta||$. It follows that $V$ has almost invariant vectors and thus there is a non-zero $V$-invariant vector, say $l$, in $K$. To show that $l$ is also $U$-invariant it is sufficient to show that for every $y\in\irrH$ there exists $x\in\irr$ such that $n_{x,y}\neq 0$. This follows from Lemma \ref{LemSurj}.
\end{proof}

\begin{corollary}
The discrete quantum groups $\widehat{A_{o}(n)}$, $\widehat{A_{u}(n)}$ and $\widehat{A_{s}(n)}$ do not have property $T$ for $n\geq 2$.
\end{corollary}

\begin{proof}
It follows directly from the preceding proposition and the following surjective morphisms:
$$A_{o}(n)\rightarrow C^{*}(\star_{i=1}^{n}\mathbb{Z}_{2}),\,\,A_{u}(n)\rightarrow C^{*}(\mathbb{F}_{n}),\,\,
A_{s}(n)\rightarrow C^{*}(\star_{i=1}^{n}\mathbb{Z}_{n_{i}}),$$
where $\sum n_{i}=n$.
\end{proof}

In the next Proposition we show that discrete quantum groups with property $T$ are unimodular.

\begin{proposition}\label{trace}
Let $\widehat{\G}$ be a discrete quantum group. If $\widehat{\G}$ has property $T$ then it is a Kac algebra, \ie{} the Haar state $\varphi$ on $\G$ is a trace.
\end{proposition}

\begin{proof}
 Suppose $\Gh$ has property $T$ and let $\Gamma$ be the discrete group introduced in Proposition \ref{Propeigenvalues}. Because Sp$(\widehat{\delta})=\Gamma\cup\{0\}$ and  $\hat{\Delta}{\widehat{\delta}}=\widehat{\delta}\ot\widehat{\delta}$, we have an injective *-homomorphism
$$\alpha\,:\,c_{0}(\Gamma)\rightarrow \coGh,\quad\alpha(f)=f(\widehat{\delta})$$
satisfying $\Delta\circ\alpha=(\alpha\ot\alpha)\circ\Delta_{\Gamma}$. By Proposition \ref{Morph}, $\Gamma$ has property $T$. It follows that $\Gamma=\{1\}$ and $\widehat{\delta}=1$. Thus $Q_{x}=1$ for all $x\in\irr$. This means that $\varphi$ is a trace.
\end{proof}

\begin{proposition}\label{Propfinitelygenerated}
Let $\widehat{\G}$ be a discrete quantum group. If $\widehat{\G}$ has property $T$ then it is finitely generated.
\end{proposition}

\begin{proof}
Let $\irr=\{x_{n}\,|\,n\in\N\}$ and $\mathcal{C}$ be the category of finite dimensional unitary representations of $\G$. For $i\in\N$ let $\mathcal{D}_{i}$ be the full subcategory of $\mathcal{C}$ generated by $(x_{0},\ldots,x_{i})$. This means that the irreducibles of $\mathcal{D}_{i}$ are the irreducible representations $u$ of $\G$ such that $u$ is equivalent to a sub-representation of $x_{k_{1}}^{\epsilon_{1}}\ot\ldots\ot x_{k_{l}}^{\epsilon_{l}}$ for $l\geq 1$, $0\leq k_{j}\leq n$, and $\epsilon_{j}$ is nothing or the contragredient. The Hilbert spaces and the morphisms are the same in $\mathcal{D}_{i}$ or in $\mathcal{D}$. Thus we have $1_{\mathcal{C}}\in\mathcal{D}_{i}$, $\mathcal{D}_{i}\ot\mathcal{D}_{i}\subset\mathcal{D}_{i}$ and $\overline{\mathcal{D}_{i}}=\mathcal{D}_{i}$. Let $\Hq_{i}$ be the compact quantum group such that $\mathcal{D}_{i}$ is the category of representation of $\Hq_{i}$. Let $U_{i}\in l^{\infty}(\Gh/\Hh_{i})\ot\mathcal{B}(l^{2}(\Gh/\Hh_{i}))$ be the quasi-regular representation of $\Gh$ modulo $\Hh_{i}$. Let $U$ be the direct sum of the $U_{i}$; this a unitary representation on $K=\bigoplus l^{2}(\Gh/\Hh_{i})$. Let us show that $U$ has almost invariant vectors. Let $E\subset\irr$ be a finite subset. There exists $i_{0}$ such that $E\subset\irrHi$ for all $i\geq i_{0}$. By Lemma \ref{Hinvariant} we have a unit vector $\xi$ in $l^{2}(\Gh/\Hh_{i})$ such that $U_{i_{0}}^{x}\eta\ot\xi=\eta\ot\xi$ for all $x\in E$ and all $\eta\in H_{x}$. Let $\tilde{\xi}=(\xi_{i})\in K$ where $\xi_{i}=0$ if $i\neq  i_{0}$ and $\xi_{i_{0}}=\xi$. Then $\tilde{\xi}$ is a unit vector in $K$ such that $U^{x}\eta\ot\tilde{\xi}=\eta\ot\tilde{\xi}$ for all $x\in E$. It follows that $U$ has an almost invariant vector. By property $T$ there exists a non-zero invariant vector $l=(l_{i})\in K$. There exists $m$ such that $l_{m}\neq 0$. Then $l_{m}$ is an invariant vector for $U_{m}$. By Lemma \ref{Invariant}, $\irr/\text{Irred}(\Hq_{m})$ is a finite set. Let $y_{1},\ldots,y_{l}$ be a complete set of representatives of $\irr/\text{Irred}(\Hq_{m})$. Then $\mathcal{C}$ is generated by $\{y_{1},\ldots,y_{l},x_{0},\ldots,x_{m},\bar{x}_{0},\ldots,\bar{x}_{m}\}$.
\end{proof}

As in the classical case, we can show that property $T$ is equivalent to the existence of a Kazhdan pair.

\begin{proposition}\label{pair}
Let $\widehat{\G}$ be a finitely generated discrete quantum group. Let $E\subset\irr$ be a finite subset with $1\in E$ such that $\mathcal{R}(\G)$ is generated by $E$. The following assertions are equivalent:
\begin{enumerate}
\item $\widehat{\G}$ has property $T$.
\item There exists $\epsilon >0$ such that every unitary representation of $\widehat{\G}$ having an $(E,\epsilon)$-invariant vector has a non-zero invariant vector.
\end{enumerate}
\end{proposition}

\begin{proof}
It is sufficient to show that $1$ implies $2$. Let $n\in\N^{*}$ and $E_{n}=\{ y\in\irr\,|\,y\subset x_{1}\ldots x_{n},\,\,x_{i}\in E\}$. Because $1\in E$, the sequence $(E_{n})_{n\in\N^{*}}$ is increasing. Let us show that $\irr=\bigcup E_{n}$. Let $r\in\irr$. Because $\mathcal{R}(\G)$ is generated by $E$, there exists a finite family of morphisms $b_{k}\in\text{Mor}(r_{k},r)$, where $r_{k}$ is a product of elements of $E$ and $\sum_{k} b_{k}b_{k}^{*}=I_{r}$. Let $L$ be the maximum of the length of the elements $r_{k}$. Because $1\in E$, we can suppose that all the $r_{k}$ are of the form $x_{1}\ldots x_{L}$ with $x_{i}\in E$. Put $t_{k}=b_{k}^{*}$. Note that $t_{k}^{*}t_{k}\in\text{Mor}(r,r)$. Because $r$ is irreducible and $\sum_{k}t_{k}^{*}t_{k}=I_{r}$, there exists a unique $k$ such that $t_{k}^{*}t_{k}=I_{r}$ and $t_{l}^{*}t_{l}=0$ if $l\neq k$. Thus $t_{k}\in\text{Mor}(r,r_{k})$ is an isometry. This means that $r\subset r_{k}=x_{1}\ldots x_{L}$, \ie{} $r\in E_{L}$.

Suppose that $\widehat{\G}$ has property $T$ and $2$ is false. Let $N=\text{Max}\{n_{x}\,|\,x\in E\}$ and  $\epsilon_{n}=\frac{1}{n^{2}\sqrt{N^{n}}}$. For all $n\in \N^{*}$ there exists a unitary representation $U_{n}$ of $\widehat{\G}$ on a Hilbert space $K_{n}$ with an $(E,\epsilon_{n})$-invariant vector but without a non-zero invariant vector. Let $\xi_{n}$ be a unit vector in $K_{n}$ which is $(E,\epsilon_{n})$-invariant. Write $U_{n}=\sum_{y\in\irr} U^{n,y}$ where $U^{n,y}$ is a unitary element in $\mathcal{B}(H_{y})\otimes\mathcal{B}(K_{n})$. Let us show the following:
\begin{equation}\label{Eq1}
||U^{n,y}\eta\otimes\xi_{n}-\eta\otimes\xi_{n}||_{H_{y}\otimes K_{n}}<\frac{1}{n}||\eta||_{H_{y}},\,\,\forall n\in\N^{*},\,\forall y\in E_{n},\,\forall\eta\in H_{y}.
\end{equation}
Let $y\in E_{n}$ and $t_{y}\in\text{Mor}(y,x_{1}\ldots x_{n})$ such that $t_{y}^{*}t_{y}=I_{y}$. Note that, by the definition of a representation and using the description of the coproduct on $\widehat{\G}$, we have  $(t_{y}\ot 1) U^{n,y}=U^{n,x_{1}}_{1,n+1}U^{n,x_{2}}_{2,n+1}\ldots U^{n,x_{n}}_{n,n+1}(t_{y}\ot 1)$ where the subscripts are used for the leg numbering notation. It follows that, for all $\eta\in H_{y}$, we have:

\begin{eqnarray*}
||U^{n,y}\eta\otimes\xi_{n}-\eta\otimes\xi_{n}||
&=&||(t_{y}\ot 1)U^{n,y}\eta\otimes\xi_{n}-(t_{y}\ot 1)\eta\otimes\xi_{n}||\\
 &=&||U^{n,x_{1}}_{1,n+1}U^{n,x_{2}}_{2,n+1}\ldots U^{n,x_{n}}_{n,n+1}t_{y}\eta\otimes\xi_{n}-t_{y}\eta\otimes\xi_{n}||\\
 &\leq &\sum_{k=1}^{n}||U^{n,x_{k}}_{k,n+1}t_{y}\eta\otimes\xi_{n}-t_{y}\eta\otimes\xi_{n}||.
\end{eqnarray*}

Let $(e^{x_{i}}_{j})_{1\leq j \leq n_{x_{i}}}$ be an orthonormal basis of $H_{x_{i}}$ and put $$t_{y}\eta=\sum\lambda_{i_{1}\ldots i_{n}}e^{x_{1}}_{i_{1}}\ot\ldots\ot e^{x_{n}}_{i_{n}}.$$ Then we have, for all $y\in E_{n}$ and $\eta\in H_{y}$,

\begin{eqnarray*}
||U^{n,y}\eta\otimes\xi_{n}-\eta\otimes\xi_{n}||
&\leq &\sum_{k}||\sum_{{i}_{1}\ldots i_{n}}\lambda_{i_{1}\ldots i_{n}}(U^{n,x_{k}}_{k,n+1}e^{x_{1}}_{i_{1}}\ot\ldots\ot e^{x_{n}}_{i_{n}}\ot\xi_{n}-e^{x_{1}}_{i_{1}}\ot\ldots\ot e^{x_{n}}_{i_{n}}\ot\xi_{n})||\\
&\leq &\sum_{k}\sum_{{i}_{1}\ldots i_{n}}|\lambda_{{i}_{1}\ldots i_{n}}|||U^{n,x_{k}}e^{x_{k}}_{i_{k}}\ot\xi_{n}-e^{x_{k}}_{i_{k}}\ot\xi_{n}||\\
&\leq & n\epsilon_{n}||t_{y}\eta||_{1},
\end{eqnarray*}

where $||t_{y}\eta||_{1}=\sum|\lambda_{{i}_{1}\ldots i_{n}}|$. Note that $||t_{y}\eta||_{1}\leq \sqrt{N^{n}}||\eta||$, thus we have

\begin{eqnarray*}
||U^{n,y}\eta\otimes\xi_{n}-\eta\otimes\xi_{n}||
&\leq &n\epsilon_{n}\sqrt{N^{n}}||\eta||\\
&\leq & \frac{1}{n}||\eta||.
\end{eqnarray*}

This proves Eq. $(\ref{Eq1})$. It is now easy to finish the proof. Let $U$ be the direct sum of the $U_{n}$. It is a unitary representation of $\widehat{\G}$ on $K=\bigoplus K_{n}$. Let $\delta>0$ and $L\subset\irr$ a finite subset. Because $\irr=\bigcup^{\uparrow} E_{n}$ there exists $n_{1}$ such that $L\subset E_{n}$ for all $n\geq n_{1}$. Choose $n\geq n_{1}$ such that $\frac{1}{n}<\delta$. Put $\xi=(0,\ldots,0,\xi_{n},0,\ldots)$ where $\xi_{n}$ appears in the n-th place. Let $x\in L$ and $\eta\in H_{x}$. We have:

\begin{eqnarray*}
||U^{x}\eta\otimes\xi-\eta\ot\xi||
&=&||U^{n,x}\eta\otimes\xi_{n}-\eta\ot\xi_{n}||\\
&\leq &\frac{1}{n}||\eta||<\delta||\eta||.
\end{eqnarray*}

Thus $U$ has almost invariant vectors. It follows from property $T$ that $U$ has a non-zero invariant vector, say $l=(l_{n})$. There is a $n$ such that $l_{n}\neq 0$ and from the U-invariance of $l$ we conclude that $l_{n}$ is $U_{n}$-invariant. This is a contradiction.
\end{proof}

Such a pair $(E,\epsilon)$ as defined Proposition \ref{pair} is called a \textit{Kazhdan pair} for $\widehat{\G}$. Let us give an obvious example of a Kazhdan pair.

\begin{proposition}
Let $\Gh$ be a finite-dimensional discrete quantum group. Then $(\irr,\sqrt{2})$ is a Kazhdan pair for $\Gh$.
\end{proposition}

\begin{proof}
If $\Gh$ is finite-dimensional then it is compact, $\varphi$ is a trace and $\hat{\varphi}$ is a normal functional. For $x\in\irr$ let $(e^{x}_{i})$ be an orthonormal basis of $H_{x}$ and $e^{x}_{ij}$ the associated matrix units. As $Q_{x}=1$, we have $\hat{\varphi}(e^{x}_{ij})=\frac{\text{dim}_{q}(x)^{2}}{n_{x}}\delta_{ij}$. Let $U\in\linfG\ot\bk$ be a unitary representation of $\Gh$ with a unit vector $\xi\in K$ such that:
$$\text{Sup}_{x\in\irr, 1\leq j\leq n_{x}}||U^{x}e^{x}_{j}\ot\xi-e^{x}_{j}\ot\xi||<\sqrt{2}.$$
Because $\hat{\varphi}(1)^{-1}(\hat{\varphi}\ot\id)(U)$ is the projection on the $U$-invariant vectors, $\tilde{\xi}=(\hat{\varphi}\ot\id)(U)\xi\in K$ is invariant. Let us show that $\tilde{\xi}$ is non-zero. Writing $U^{x}=\sum e^{x}_{ij}\ot U^{x}_{ij}$ with $U^{x}_{ij}\in\bk$, we have:
$$||U^{x}e^{x}_{j}\ot\xi-e^{x}_{j}\ot\xi||^{2}=2-2\text{Re}\langle U^{x}_{jj}\xi,\xi\rangle,\quad\text{for all}\,\,x\in\irr,\,\,1\leq j\leq n_{x}.$$
It follows that $\text{Re}\langle U^{x}_{jj}\xi,\xi\rangle>0$ for all $x\in\irr$ and all $1\leq j\leq n_{x}$. Thus,
$$\text{Re}\langle\tilde{\xi},\xi\rangle=\sum_{x,i,j}\text{Re}(\hat{\varphi}(e^{x}_{ij})\langle U^{x}_{ij}\xi,\xi\rangle)=\sum_{x,i}\frac{\text{dim}_{q}(x)^{2}}{n_{x}}\text{Re}(\langle U^{x}_{ii}\xi,\xi\rangle)>0.$$
\end{proof}

\begin{remark}
It is easy to see that a discrete quantum group is amenable and has property $T$ if and only if it is finite-dimensional. Indeed, the existence of almost invariant vectors for the regular representation is equivalent with amenability and it is well known that a discrete quantum group is finite dimensional if and only if the regular representation has a non-zero invariant vector. Moreover the previous proposition implies that all finite-dimensional discrete quantum groups have property $T$.
\end{remark}

The main result of this paper is the following.

\begin{theorem}
Let $\widehat{\G}$ be discrete quantum group such that $L^{\infty}(\G)$ is an infinite dimensional factor. The following assertions are equivalent :
\begin{enumerate}
\item $\widehat{\G}$ has property $T$.
\item $L^{\infty}(\G)$ is a $\rm{II}_{1}$ factor with property $T$.
\end{enumerate}
\end{theorem}

\begin{proof}
We can suppose that $\G$ is reduced, $C(\G)\subset\mathcal{B}(L^{2}(\G))$ and $\mathbb{V}\in\linfG\ot\L^{\infty}(\G)$. We denote by $M$ the von Neumann algebra $L^{\infty}(\G)$. For each $x\in\irr$ we choose an orthonormal basis $(e^{x}_{i})_{1\leq i \leq n_{x}}$ of $H_{x}$. When $\varphi$ is a trace we take $e^{\bar{x}}_{i}=J_{x}(e^{x}_{i})$. We put $u^{x}_{ij}=(\omega_{e^{x}_{j},e^{x}_{i}}\otimes\id)(u^{x})$.

$1\Rightarrow 2$ : Suppose that $\widehat{\G}$ has property $T$. By Proposition \ref{trace}, M is finite factor. Thus, it is a $\rm{II}_{1}$ factor. Let $(E,\epsilon)$ be a Kazhdan pair for $\widehat{\G}$. Let $K\in\mathcal{C}(M)$ with morphisms $\pi_{l}\,:\,M\rightarrow\bk$ and $\pi_{r}\,:\,M^{op}\rightarrow\bk$. Let $\delta=\frac{\epsilon}{\text{Max}\{n_{x}\sqrt{n_{x}},\,x\in E\}}$. Suppose that there exists a unit vector $\xi^{'}\in K$ such that:
$$||u^{x}_{ij}\xi^{'}-\xi^{'}u^{x}_{ij}||<\delta,\quad\forall x\in E,\,\,\forall\,1\leq i,j\leq  n_{x}.$$
Define $U=(\id\otimes\pi_{r})(\V^{*})(\id\otimes\pi_{l})(\V)$. Because $\mathbb{V}$ is a unitary representation of  $\Gh$ and $\pi_{r}$ is an anti-homomorphism, it is easy to check that $U$ is a unitary representation of $\Gh$ on $K$. Moreover, for all $x\in E$, we have:

\begin{eqnarray*}
||U^{x}e^{x}_{i}\ot\xi^{'}-e^{x}_{i}\ot\xi^{'}||
&=&||(\id\otimes\pi_{l})(u^{x})e^{x}_{i}\ot\xi^{'}-(\id\otimes\pi_{r})(u^{x})e^{x}_{i}\ot\xi^{'}||\\
&=&||\sum_{k=1}^{n_{x}}e^{x}_{k}\ot(u^{x}_{ki}\xi^{'}-\xi^{'}u^{x}_{ki})||\\
&\leq &\sum_{k=1}^{n_{x}}||e^{x}_{k}\ot(u^{x}_{ki}\xi^{'}-\xi^{'}u^{x}_{ki})||\\
&< &n_{x}\delta\leq\frac{\epsilon}{\sqrt{n_{x}}}.
\end{eqnarray*}

It follows easily that for all $x\in E$ and all $\eta\in H_{x}$ we have $||U^{x}\eta\ot\xi^{'}-\eta\ot\xi^{'}||<\epsilon||\eta||$. Thus there exists a non-zero $U$-invariant vector $\xi\in K$. It is easy to check that $\xi$ is a central vector.

$2\Rightarrow 1$ : Suppose that $M$ is a $\rm{II}_{1}$ factor with property $T$ and let $\epsilon>0$ and $b_{1},\ldots, b_{n}\in M$ be as in Proposition \ref{propCJ}. Let $\varphi$ be the Haar state on $\G$. By \cite{Fima}, Theorem 8, $\varphi$ is the unique tracial state on $M$. We can suppose that $||b_{i}||_{2}=1$. Using the classical G.N.S. construction $(L^{2}(\G),\Omega)$ for $\varphi$ we have, for all $a\in M$,
$$a\Omega=\sum_{x,k,l}n_{x}\varphi((u^{x}_{kl})^{*}a)u^{x}_{kl}\Omega.$$
In particular, $||b_{i}||_{2}^{2}=\sum n_{x}|\varphi((u^{x}_{kl})^{*}b_{i})|^{2}=1$. Fix $\delta>0$ then there exists a finite subset $E\subset\irr$ such that, for all $1\leq i\leq n$,
$$\sum_{x\notin E, k,l} n_{x}|\varphi((u^{x}_{kl})^{*}b_{i})|^{2}<\delta^{2}.$$
Let $U$ be a unitary representation of $\Gh$ on $K$ having almost invariant vectors and $\xi\in K$ an $(E,\delta)$-invariant unit vector. Turn $L^{2}(\G)\ot K$ into a correspondence from $M$ to $M$ using the morphisms $\pi_{l}\,:\,M\rightarrow\mathcal{B}(L^{2}(\G)\ot K)$, $\pi_{l}(a)=U(a\ot 1) U^{*}$ and $\pi_{r}\,:\,M^{op}\rightarrow\mathcal{B}(L^{2}(\G)\ot K)$, $\pi_{r}(a)=Ja^{*}J\ot 1$, where $J$ is the modular conjugation of $\varphi$. Let $\widehat{\xi}=\Omega\ot\xi$. It is easy to see that $\pi_{l}(u^{x}_{kl})=\sum_{s}u^{x}_{ks}\ot U^{x}_{sl}$ and, for all $a\in M$,
$$a\widehat{\xi}=\sum n_{x}\varphi((u^{x}_{kl})^{*}a)u^{x}_{ks}\Omega\ot U^{x}_{sl}\xi.$$
Note that, because $\varphi$ is a trace, $\Omega$ is a central vector in $L^{2}(\G)$ and we have, for all $a\in M$, $\widehat{\xi}a=a\Omega\ot\xi$. It follows that, for all $1\leq i\leq n$, we have

\begin{eqnarray*}
||b_{i}\widehat{\xi}-\widehat{\xi}b_{i}||^{2}
&=&|| \sum_{x,k,l,s}n_{x}\varphi((u^{x}_{kl})^{*}b_{i})u^{x}_{ks}\Omega\ot U^{x}_{sl}\xi-\sum_{x,k,l}n_{x}\varphi((u^{x}_{kl})^{*}b_{i})u^{x}_{kl}\Omega\ot\xi||^{2}\\
&=&|| \sum_{x,k,l}n_{x}\varphi((u^{x}_{kl})^{*}b_{i})\left(\sum_{s}u^{x}_{ks}\Omega\ot U^{x}_{sl}\xi-u^{x}_{kl}\Omega\ot\xi\right)||^{2}\\
&=&|| \sum_{x,k,l}\sqrt{n_{x}}\varphi((u^{x}_{kl})^{*}b_{i})\left(\sum_{s} e^{x}_{s}\ot J_{x}(e^{x}_{k})\ot U^{x}_{sl}\xi-e^{x}_{l}\ot J_{x}(e^{x}_{k})\ot\xi\right)||^{2}\\
&=&|| \sum_{x,k,l}\sqrt{n_{x}}\varphi((u^{x}_{kl})^{*}b_{i})J_{x}(e^{x}_{k})\ot\left(\sum_{s} e^{x}_{s}\ot U^{x}_{sl}\xi-e^{x}_{l}\ot\xi\right)||^{2}\\
&=&|| \sum_{x,k,l}\sqrt{n_{x}}\varphi((u^{x}_{kl})^{*}b_{i})J_{x}(e^{x}_{k})\ot( U^{x}e^{x}_{l}\ot\xi-e^{x}_{l}\ot\xi)||^{2}\\
&=&  \sum_{x,k}n_{x}||\sum_{l}\varphi((u^{x}_{kl})^{*}b_{i})( U^{x}e^{x}_{l}\ot\xi-e^{x}_{l}\ot\xi)||^{2}\\
&=& \sum_{x,k}n_{x}||U^{x}\eta^{x}_{k}\ot\xi-\eta^{x}_{k}\ot\xi)||^{2},
\,\,\,\text{where}\,\, \eta^{x}_{k}=\sum_{l}\varphi((u^{x}_{kl})^{*}b_{i})e^{x}_{l}\\
&=& \sum_{x\in E,k}n_{x}||U^{x}\eta^{x}_{k}\ot\xi-\eta^{x}_{k}\ot\xi)||^{2}
+\sum_{x\notin E,k}n_{x}||U^{x}\eta^{x}_{k}\ot\xi-\eta^{x}_{k}\ot\xi)||^{2}\\
&<&  \delta^{2}\sum_{x\in E,k}n_{x}||\eta^{x}_{k}||^{2}
+4\sum_{x\notin E,k}n_{x}||\eta^{x}_{k}||^{2}\\
&<& \delta^{2}\sum_{x\in E,k,l}n_{x}|\varphi((u^{x}_{kl})^{*}b_{i})|^{2}
+4\sum_{x\notin E,k,l}n_{x}|\varphi((u^{x}_{kl})^{*}b_{i})|^{2}\\
&<&  \delta^{2} + 4 \delta^{2}= 5\delta^{2}.\\
\end{eqnarray*}
By Proposition \ref{propCJ}, for $\delta$ small enough, there exists a central unit  vector $\widehat{\eta}\in L^{2}(\G)\ot K$ with $||\widehat{\eta}-\widehat{\xi}||<\sqrt{5}C\delta$. Let $P$ be the orthogonal projection on $\C\Omega$. If $\delta$ is small enough then there is a non-zero $\eta\in K$ such that $(P\ot 1)\widehat{\eta}=\Omega\ot\eta$. Write $\widehat{\eta}=\sum_{y,s,t}e^{y}_{t}\otimes e^{\bar{y}}_{s}\otimes\eta^{y}_{s,t}$ where $\eta^{y}_{s,t}\in K$ and $\eta^{1}=\eta$. We have, for all $x\in\irr$ and all $1\leq i,j\leq n_{x}$, $\pi_{l}(u^{x}_{ij})\widehat{\eta}=\pi_{r}(u^{x}_{ij})\widehat{\eta}$. This means:
\begin{eqnarray}\label{Eqinv}
\sum_{k,y,t,s}u^{x}_{ik}(e^{y}_{t}\otimes e^{\bar{y}}_{s})\otimes U^{x}_{kj}\eta^{y}_{s,t}
=\sum_{y,t,s}J(u^{x}_{ij})^{*}J(e^{y}_{t}\otimes e^{\bar{y}}_{s})\otimes\eta^{y}_{s,t}.
\end{eqnarray}
Let $Q$ be the orthogonal projection on $H_{x}\otimes H_{\bar{x}}$. Using
$$u^{x}_{ik}(e^{y}_{t}\otimes e^{\bar{y}}_{s})\subset\bigoplus_{z\subset x\ot y}H_{z}\ot H_{\bar{z}},$$
and $x\subset x\ot y$ if and only if $y=1$, we find:
$$Qu^{x}_{ik}(e^{y}_{t}\otimes e^{\bar{y}}_{s})=\delta_{y,1}\frac{1}{\sqrt{n_{x}}}e^{x}_{k}\otimes e^{\bar{x}}_{i}.$$
Using the same arguments and the fact that $J=\bigoplus(J_{x}\otimes J_{\bar{x}})$ we find:
$$QJ(u^{x}_{ij})^{*}J(e^{y}_{t}\otimes e^{\bar{y}}_{s})=\delta_{y,1}\frac{1}{\sqrt{n_{x}}}e^{x}_{j}\otimes e^{\bar{x}}_{i}.$$
Applying $Q\ot 1$ to Eq. $(\ref{Eqinv})$ we obtain:
$$\sum_{k}e^{x}_{k}\otimes e^{\bar{x}}_{i}\otimes U^{x}_{kj}\eta=e^{x}_{j}\otimes e^{\bar{x}}_{i}\otimes\eta,\quad\text{for all}\,\,x\in\irr,\,\,1\leq i,j\leq n_{x}.$$
Thus, for all $x\in\irr$ and all $1\leq j\leq n_{x}$, we have:
$$U^{x}(e^{x}_{j}\otimes\eta)=\sum_{k}e^{x}_{k}\otimes U^{x}_{kj}\eta=e^{x}_{j}\otimes\eta.$$
Thus $\eta$ is a non-zero $U$-invariant vector.
\end{proof}

The preceding theorem admits the following corollary about the persistance of property $T$ by twisting.

\begin{corollary}
Let $\G$ be a compact quantum group such that $L^{\infty}(\G)$ is an infinite dimensional factor. Suppose that $K$ is an abelian co-subgroup of $\G$ (see \cite{FimaVain}). Let $\sigma$ be a continuous bicharacter on $\widehat{K}$ and denote by $\G^{\sigma}$ the twisted quantum group. If $\Gh$ has property $T$ then $\widehat{\G^{\sigma}}$ is a discrete quantum group with property $T$.
\end{corollary}

\begin{proof}
If $\Gh$ has property $T$ then the Haar state $\varphi$ on $\G$ is a trace. Thus the co-subgroup $K$ is stable (in the sense of \cite{FimaVain}) and the Haar state $\varphi_{\sigma}$ on $\G^{\sigma}$ is the same, \ie{} $\varphi=\varphi_{\sigma}$. It follows that $\G^{\sigma}$ is a compact quantum group with $L^{\infty}(\G^{\sigma})=L^{\infty}(\G)$. Thus $L^{\infty}(\G^{\sigma})$ is a $\rm{II}_{1}$ factor with property $T$ and $\widehat{\G_{\sigma}}$ has property $T$.
\end{proof}

\begin{example}
The group $SL_{2n+1}(\Z)$ is I.C.C. and has property $T$ for all $n\geq 1$. Let $K_{n}$ be the subgroup of diagonal matrices in $SL_{2n+1}(\Z)$. We have $K_{n}=\Z_{2}^{2n}=\langle t_{1},\ldots,t_{2n}\,|\,t_{i}^{2}=1\,\,\forall i,\,\,t_{i}t_{j}=t_{j}t_{i}\,\,\forall i,j\rangle$ and $K_{n}$ is an abelian co-subgroup of $\G_{2n+1}=(C^{*}(SL_{2n+1}(\Z)),\Delta)$. Consider the following bicharacter on $\widehat{K_{n}}=K_{n}$: $\sigma$ is the unique bicharacter such that $\sigma(t_{i},t_{j})=-1$ if $i\leq j$ and $\sigma(t_{i},t_{j})=1$ if $i>j$. By the preceding Corollary, the twisted quantum group $\widehat{\G_{2n+1}^{\sigma}}$ has property $T$ for all $n\geq 1$. When $n$ is even, $SL_{n}(\Z)$ is not I.C.C. and  $I$ and $-I$ lie in the centre of $SL_{n}(\Z)$. We consider the group $PSL_{n}(\Z)=SL_{n}(\Z)/\{I,-I\}$ in place of $SL_{n}(\Z)$ in the even case. It is well known that $PSL_{2n}(\Z)$ is I.C.C. and has property $T$ for $n\geq 2$. The group of diagonal matrices in $SL_{2n}(\Z)$ is $\Z_{2}^{2n-1}$ which contains $\{I,-I\}$. We consider the following abelian subgroup of $PSL_{2n}(\Z)$: $L_{n}=\Z_{2}^{2n-1}/\{I,-I\}=\Z_{2}^{2n-2}=K_{n-1}$ and the same bicharacter $\sigma$ on $K_{n-1}$. Let $\G_{2n}=(C^{*}(PSL_{2n}(\Z)),\Delta)$. By the preceding Corollary, the twisted quantum group $\widehat{\G_{2n}^{\sigma}}$ has property $T$ for all $n\geq 2$.
\end{example}

\end{document}